\documentclass[a4paper,fleqn]{article}
\usepackage{amssymb,amsthm,amsfonts}
\usepackage{amsmath}
\newtheorem{theorem}{Theorem}

\newtheorem{corollary}{Corollary}

\newtheorem{remark}{Remark}

\numberwithin{equation}{section}
\numberwithin{lemma}{section}
\numberwithin{theorem}{section}
\numberwithin{corollary}{section}

\allowdisplaybreaks
\begin{document}
\setcounter{page}{1}

\title{On the  incomplete Srivastava's  triple  hypergeometric matrix  functions}

\author{Ashish Verma\footnote{Corresponding author}
\\ 
Department of Mathematics\\ Prof. Rajendra Singh (Rajju Bhaiya)\\ Institute of Physical Sciences for Study and Research \\  V.B.S. Purvanchal University, Jaunpur  (U.P.)- 222003, India\\
vashish.lu@gmail.com
\\[10pt]
}

\maketitle
\begin{abstract}
The paper proposes to  introduce  incomplete  Srivastava's  triple hypergeometric  matrix functions  through application of  the  incomplete Pochhammer matrix   symbols. We also derive   certain properties  such as  matrix differential equation, integral formula, reduction formula,  recursion formula,  recurrence  relation and differentiation formula of the incomplete Srivastava's  triple hypergeometric  matrix functions.
 \\[12pt]
\end{abstract}

\section{Introduction}Recently, Srivastava {\em et al}. \cite{HM}  have studied incomplete Pochhammer symbols  and  incomplete hypergeometric functions and discussed   applications of these functions in communication theory, probability theory and groundwater pumping modelling. Cetinkaya \cite{AC}   introduced the incomplete second Appell hypergeometric functions and obtained certain properties of these functions.  Also, recently introduced incomplete Srivastava's triple hypergeometrics and  investigated certain properties of the incomplete Srivastava's triple hypergeometrics  \cite{jc, jc1}.  Srivastava {\em et al}. \cite{SSK} have obtained several interesting properties of the incomplete $H$-functions. On his work on hypergeometric functions of three variables, Srivastava \cite{HM1, HM2} noticed the existence of three additional complete triple hypergeometric functions of the second-order. These functions are known in literature as Srivastava’s triple hypergeometric functions $H_A$, $H_B$ and $H_C$ and are given in \cite{SK, SM}.

For a wide variety of other explorations involving incomplete hypergeometric functions in several variables, the interested reader may be referred to several recent papers  \cite{8r1, jc, jc1,  8r2, 8r5, 8r3, 8r4}.

The matrix theory has become pervasive to  almost every area of Mathematics, in general  and in orthogonal polynomial and special functions, in particular.  The matrix analogue of the Gauss hypergeometric function was introduced by  J\'odar and  Cort\'es \cite{LC}, particularly the hypergeometric matrix function, plays a very important role in solving numerous problems of mathematical physics, engineering and mathematical sciences \cite{ p2,p3, p1, p5, p4}.  Quite recently, the incomplete hypergeometric matrix functions was introduced by  Abdalla \cite{Ab}.  Within the frame, they discuss some fundamental  properties of these functions.  In a similar vein   Bakhet {\em et al}. \cite{AY} introduced the Wright hypergeometric matrix functions and the incomplete Wright Gauss hypergeometric matrix functions and  discuss some properties of these functions. The paper is in succession of the work we had already attempted and successfully completed \cite{A1, A2, A3} that had introduced the  incomplete first, second and fourth Appell hypergeometric matrix functions and studied some basic properties such as matrix differential equation, integral formula, recursion formula, recurrence relation and differentiation formula of these functions.

Here is elucidated a sectionwise distribution of present work.  In Section~2, we list basic definitions that are needed in  sequel.  In Section~3, we first  introduce two incomplete Srivastava's  triple hypergeometric  matrix functions $\gamma_{\mathcal{A}}^{H}$ and $\Gamma_{\mathcal{A}}^{H}$ by using properties of Pochhaommer matrix symbol. Some properties such as matrix differential equation, integral representation, reduction formula, recursion formula,  recurrence relation and differentitation formula are also derived. In Section~4, we introduce  incomplete Srivastava's  triple hypergeometric  matrix functions $\gamma_{\mathcal{B}}^{H}$ and $\Gamma_{\mathcal{B}}^{H}$ by using properties of Pochhaommer matrix symbol and investigate several properties of each of these incomplete Srivastava’s triple hypergeometric functions. Finally, in Section 5, we define  incomplete Srivastava's  triple hypergeometric  matrix functions $\gamma_{\mathcal{C}}^{H}$ and $\Gamma_{\mathcal{C}}^{H}$ by using properties of Pochhaommer matrix symbol and derive certain properties of each of these  functions. The work at hand also attempts to  present connections between these matrix functions and  Bessel and Laguerre  matrix functions.

\section{Preliminaries}
Throughout this paper,  let $\mathbb{C}^{r\times r}$ be the vector space of $r$-square matrices with complex entries. For any matrix $A\in \mathbb{C}^{r\times r}$, its spectrum $\sigma(A)$ is the set of eigenvalues of $A$. 
 A square matrix $A$ in $\mathbb{C}^{r\times r}$  is said to be positive stable if $\Re(\lambda)>0$ for all $\lambda\in\sigma(A)$. 

Let $A$ be positive stable matrix in $\mathbb{C}^{r\times r}$. The Gamma matrix function $\Gamma(A)$ is defined as follows \cite{LC1}:
\begin{align}
\Gamma(A)=\int_{0}^{\infty} e^{-t} t^{A-I} dt; \hskip1cm t^{A-I}= \exp((A-I)\ln t), \label{g1}
\end{align}
where $I$ being the $r$-square identity matrix.

Let $A$ be positive stable matrix in $\mathbb{C}^{r\times r}$ and let $x$ be a positive real number.
Then the  incomplete gamma matrix  functions $\gamma(A,x)$ and $\Gamma(A,x)$ are defined by \cite{Ab}
\begin{align}\gamma(A,x)=\int_{0}^{x} e^{-t} t^{A-I}dt \label{1eq4}
\end{align}
and
\begin{align}
 \Gamma(A,x)= \int_{x}^{\infty} e^{-t} t^{A-I}dt\,,\label{1eq5}
\end{align}
respectively and satisfy the following decomposition formula:
\begin{align}\gamma(A,x)+\Gamma(A,x)=\Gamma(A).\label{1eq6}
\end{align}
Let  $A$  be a matrix in $\mathbb{C}^{r\times r}$ and let $x$ be a positive real number. Then the incomplete Pochhammer matrix symbols $(A;x)_{n}$ and $[A; x]_{n}$ are defined as follows  \cite{Ab}
\begin{align}
(A; x)_{n}= \gamma(A+nI, x) \,\Gamma^{-1}(A)\label{1eq7}
\end{align}
and 
\begin{align}
[A; x]_{n}= \Gamma(A+nI, x)\, \Gamma^{-1}(A).\label{1eq8}
\end{align}
In idea of (\ref{1eq6}), these incomplete Pochhammer matrix symbols $(A; x)_{n}$ and $[A; x]_{n}$ complete the following decomposition relation
\begin{align}(A; x)_{n}+[A; x]_{n}= (A)_{n}.\label{1eq9}
\end{align}
where $(A)_{n}$ is the Pochhammer  symbol given in \cite{LC}.

Let $A$, $B$ and $C$ be matrices in  $\mathbb{C}^{r\times r}$ such that $C+kI$ is invertible for all integers $k \geq0.$ The incomplete Gauss hypergeometric matrix  functions  are defined by  \cite{Ab}
\begin{align}{ _2\gamma_{1}}\Big[(A; x), B; C; z\Big]= \sum_{n=0}^{\infty}(A;x)_{n} (B)_{n}(C)_{n}^{-1}\frac{z^{n}}{n!}\label{11eq10}
\end{align}
and
\begin{align}{_2\Gamma_{1}}\Big[[A; x], B; C; z\Big]= \sum_{n=0}^{\infty}[A;x]_{n} (B)_{n}(C)_{n}^{-1}\frac{z^{n}}{n!}.\label{1eq10}
\end{align}

The  Bessel matrix function is defined by \cite{J1, J2, J3}: 
\begin{align}J_{A}(z)=\sum_{m\geq 0}^{\infty}\frac{(-1)^{m}\,\,\Gamma^{-1}(A+(m+1)I)}{m!}\Big(\frac{z}{2}\Big)^{A+2mI},\label{j1}
\end{align}
where $A+kI$ is invertible for all integers $k\geq 0$. Therefore, the modified Bessel matrix functions are introduced in \cite{J3}  in the form 
\begin{align}&I_{A}= e^\frac{-Ai\pi}{2} J_{A}(z  e^\frac{i\pi}{2}); \,\,\, -\pi<arg(z)<\frac{\pi}{2},\notag\\
&I_{A}= e^\frac{Ai\pi}{2} J_{A}(z  e^\frac{-i\pi}{2}); \,\,\, -\frac{\pi}{2}<arg(z)<\pi.\label{j2}
\end{align}

The Laguerre matrix polynomial is defined by \cite{LR}
\begin{align}
L_{n}^{(A, \lambda)}(z)=\sum_{k=0}^{n}\frac{(-1)^{k}{\lambda}^{k}}{k!\, (n-k)!}(A+I)_{n}{[(A+I)_{k}]}^{-1} z^{k}.
\end{align}
It follows that for $ \lambda=1$, we have
\begin{align}L_{n}^{(A)}(z)=\frac{(A+I)_{n}}{n!}\, _{1}F_{1}(-nI; A+I; z).
\end{align}
\section{The  incomplete Srivastava's  triple  hypergeometric matrix  functions $\gamma_{\mathcal{A}}^{H}$ and $\Gamma_{\mathcal{A}}^{H}$}
This section deals with  the incomplete Srivastava's  triple  hypergeometric matrix  functions $\gamma_{\mathcal{A}}^{H}$ and $\Gamma_{\mathcal{A}}^{H}$ as follows: 
\begin{align}
&\gamma_{\mathcal{A}}^{H}[(A; x), B, B'; C, C'; z_1, z_2, z_3]\notag\\&=\sum_{m, n, p\geq 0}(A;x)_{m+p} (B)_{m+n} (B')_{n+p} (C)^{-1}_{m}(C')^{-1}_{n+p}\frac{z_{1}^{m}z_2^{n}z_3^{p}}{m! n! p!},\label{2eq1}\\
&\Gamma_{\mathcal{A}}^{H}[(A; x), B, B'; C, C' ; z_1, z_2, z_3]\notag\\&=\sum_{m, n, p\geq 0}[A;x]_{m+p} (B)_{m+n} (B')_{n+p} (C)^{-1}_{m}(C')^{-1}_{n+p}\frac{z_{1}^{m}z_2^{n}z_3^{p}}{m! n! p!}.\label{2eq2}
\end{align}
where $A$, $B$, $B'$, $C$,  $C'$  are  positive stable and commutative  matrices in  $\mathbb{C}^{r\times r}$ such that $C+kI$ and  $C'+kI$  are invertible for all integers $k\geq0$.\\
From (\ref{1eq9}), we have the following decomposition formula
\begin{align}
&\gamma_{\mathcal{A}}^{H}[(A; x), B, B'; C,C' ; z_1, z_2, z_3]+ \Gamma_{\mathcal{A}}^{H}[(A; x), B, B'; C,C' ; z_1, z_2, z_3]\notag\\&= H_{\mathcal{A}}[A, B, B'; C, C' ; z_1, z_2, z_3].\label{2eq3}
\end{align}
where $H_{\mathcal{A}}[A, B, B'; C,C' ; z_1, z_2, z_3]$ is the Srivastava's  triple  hypergeometric matrix  functions \cite{RD1}.

\begin{remark}
It is quite evident  that the special cases of (\ref{2eq1}) and (\ref{2eq2}) when $z_2 = 0$ reduces to the
known incomplete families of the second Appell hypergeometric  matrix functions \cite{A3}. Also, the special
cases of (\ref{2eq1}) and (\ref{2eq2}) when $z_2 = 0$ and  $z_3 = 0$ or $z_1 = 0$ are seen to yield the known incomplete families of Gauss hypergeometric  matrix functions \cite{Ab}.\end{remark}

If one can successfully  discuss the properties and characteristics of \\$\gamma_{\mathcal{A}}^{H}[(A; x), B, B'; C, C' ; z_1, z_2, z_3]$, it will evidently  be sufficient to determine the properties of $\Gamma_{\mathcal{A}}^{H}[(A; x), B, B'; C, C' ; z_1, z_2, z_3]$ according to the decomposition formula (\ref{2eq3}). 
\begin{theorem}For commuting matrices $A$, $B$, $B'$, $C$  and $C'$ in $\mathbb{C}^{r\times r}$ such that $C+kI$  and $C'+kI$ is invertible for all integers $k\geq0.$ Then the function defined by 
${\mathcal{T}_1}={\mathcal{T}_1}(z_1, z_2, z_3) =\gamma_{\mathcal{A}}^{H}[(A; x), B, B'; C,C' ; z_1, z_2, z_3]+\Gamma_{\mathcal{A}}^{H}[(A; x), B, B'; C,C' ; z_1, z_2, z_3]$  satisfies the following system of partial differential equations:
\begin{align}
&\Big[z_1\frac{\partial}{\partial z_1}{(z_1\frac{\partial}{\partial z_1}+C-I)- z_1(z_1\frac{\partial}{\partial z_1}+z_3\frac{\partial}{\partial z_3}+A)(z_1\frac{\partial}{\partial z_1}+z_2\frac{\partial}{\partial z_2}+B)}\Big] \mathcal{T}=O,\label{2eq4}\\
&\Big[z_2\frac{\partial}{\partial z_2}{(z_2\frac{\partial}{\partial z_2}+ z_3\frac{\partial}{\partial z_3}+C'-I)- z_2(z_2\frac{\partial}{\partial z_2}+z_1\frac{\partial}{\partial z_1}+B)(z_2\frac{\partial}{\partial z_2}+z_3\frac{\partial}{\partial z_3}+B')}\Big] \mathcal{T}=O,\label{m1}\\
&\Big[z_3\frac{\partial}{\partial z_3}{(z_2\frac{\partial}{\partial z_2}+z_3\frac{\partial}{\partial z_3}+C'-I)- z_3(z_1\frac{\partial}{\partial z_1}+z_3\frac{\partial}{\partial z_3}+A)(z_2\frac{\partial}{\partial z_2}+z_3\frac{\partial}{\partial z_3}+B')}\Big] \mathcal{T}=O.\label{2eq5}
\end{align}
\end{theorem}
\begin{proof}(\ref{2eq3}) suceeds into the following proof cojoined with $H_{\mathcal{A}}[A, B, B'; C, C' ; z_1, z_2, z_3]$ which adequately satisfies the system of the  matrix  differential equations given by \cite{RD1}.
\end{proof}
\begin{theorem} Let $A$, $B$, $B'$, $C$ and $C'$ be  positive stable and commutative matrices in  $\mathbb{C}^{r\times r}$ such that $A+kI$, $B+kI$, $C+kI$ and  $C'+kI$  are  invertible for all integers $k\geq0$. Then the following integral representation for $\Gamma_{\mathcal{A}}^{H}$ in (\ref{2eq2}) holds true:
\begin{align}
&\Gamma_{\mathcal{A}}^{H}[(A; x), B, B'; C,C'; z_1, z_2, z_3]= \Gamma^{-1}(A) \Gamma^{-1}(B)\notag\\&\times \Big[\int_{x}^{\infty} \int_{0}^{\infty}e^{-t-s}\, t^{A-I} s^{B-I}  {_{0}F_{1}}(-; C; z_1 st)\, {_{1}F_{1}}(B'; C' ; z_2 s+ z_3 t) dt ds\Big],\label{2eq6}
\end{align}
where ${_{0}F_{1}}(-; C; z_1)$ is the  Gauss hypergeometric  matrix function of one denomenator and ${_{1}F_{1}}(B; C ; z_1)$ is the Kummer hypergeometric  matrix function.
\end{theorem}
\begin{proof}By replacing the incomplete Pochhammer matrix symbol $[A; x]_{m+p}$  in (\ref{1eq5}) and (\ref{1eq8}) and the   Pochhammer matrix symbol $(B)_{m+n}$  by its integral representation in (\ref{2eq2}), one attain the following equation
\begin{align}
&\Gamma_{\mathcal{A}}^{H}[(A; x), B, B'; C,C' ; z_1, z_2, z_3]\notag\\&= \Gamma^{-1}(A) \sum_{m, n, p\geq0}^{}\Big(\int_{x}^{\infty}e^{-t}\, t^{A+(m+p-1)I}dt\Big)\notag\\
&\times  (B)_{m+n} (B')_{n+p} (C)^{-1}_{m} (C')^{-1}_{n+p}\frac{z_1^{m} z_{2}^{n} z_{3}^{p}}{m! n! p!},\notag\\
&=\Gamma^{-1}(A) \Gamma^{-1}(B)\sum_{m, n, p\geq0}^{}\Big(\int_{x}^{\infty}\int_{0}^{\infty} e^{-t-s} t^{A+(m+p-1)I} s^{B+(m+n-1)I}dt ds\Big)\notag\\
&\times  (B')_{n+p} (C)^{-1}_{m} (C')^{-1}_{n+p}\frac{z_1^{m} z_{2}^{n} z_{3}^{p}}{m! n! p!},\notag\\
&=\Gamma^{-1}(A) \Gamma^{-1}(B)\Big(\int_{x}^{\infty}\int_{0}^{\infty} e^{-t-s} t^{A-I} s^{B-I}dt ds\Big)\notag\\
&\times \Big(\sum_{m\geq 0}^{}(C)^{-1}_{m}\frac{(z_1 st)^{m}}{m!}\Big)\Big(\sum_{n, p\geq 0}^{} (B')_{n+p} (C')^{-1}_{n+p}\frac{ (z_{2} s)^{n} (z_{3} t)^{p}}{n! p!}.\label{l1}
\end{align}
Taking into account the summation formula \cite{SM}
\begin{align}
\sum_{N\geq 0}^{}f(N) \frac{(z_1+z_2)^{N}}{N!}=\sum_{m, n\geq 0}^{}f(m+n)\frac{z_{1}^{m}}{m!}\frac{z_{2}^{n}}{n!}.\label{l2}
\end{align}
we get (\ref{2eq6}).
Hereby stands completed  the proof of  (\ref{2eq6}).
\end{proof}
\begin{corollary}
The following double integral representations hold true:
\begin{align}&\Gamma_{\mathcal{A}}^{H}[(A; x), B, -mI; C, C'+I ; z_1, z_2, z_3]\notag\\
&=m![(C'+I)_{m}]^{-1} \Gamma^{-1}(A) \Gamma^{-1}(B)\notag\\&\times \Big[\int_{x}^{\infty} \int_{0}^{\infty}e^{-t-s}\, t^{A-I} s^{B-I}  {_{0}F_{1}}(-; C; z_1 st)\, L^{(C')}_{m}(z_2 s+ z_3 t) dt ds\Big],
\end{align}
where $A$, $B$, $B'$, $C$ and $C'$ are  positive stable and commutative matrices in  $\mathbb{C}^{r\times r}$ such that $A+kI$, $B+kI$, $C+kI$ and  $C'+kI$  are  invertible for all integers $k\geq0$. 
\end{corollary}
\begin{corollary}The following double integral representations hold true:
\begin{align}&\Gamma_{\mathcal{A}}^{H}[(A; x), B, B'; C+I,C' ; -z_1, z_2, z_3]\notag\\
&=z_{1}^{\frac{-C}{2}}\, \Gamma^{-1}(A) \Gamma^{-1}(B)\Gamma(C+I)\notag\\
&\times  \int_{x}^{\infty}\int_{0}^{\infty} e^{-t-s} \,t^{A-\frac{C}{2}-I} s^{B-\frac{C}{2}-I} J_{C}(2\sqrt{z_1 st}) {_{1}F_{1}}(B'; C' ; z_2 s+ z_3 t) dt ds;
\end{align}

\begin{align}&\Gamma_{\mathcal{A}}^{H}[(A; x), B, B'; C+I,C' ;  z_1, z_2, z_3]\notag\\
&=z_{1}^{\frac{-C}{2}}\, \Gamma^{-1}(A) \Gamma^{-1}(B)\Gamma(C+I)\notag\\
&\times  \int_{x}^{\infty}\int_{0}^{\infty} e^{-t-s} \,t^{A-\frac{C}{2}-I} s^{B-\frac{C}{2}-I} I_{C}(2\sqrt{z_1 st}) {_{1}F_{1}}(B'; C' ; z_2 s+ z_3 t) dt ds,
\end{align}
where $A$, $B$, $B'$, $C$ and $C'$ are  positive stable and commutative matrices in  $\mathbb{C}^{r\times r}$ such that $A+kI$, $B+kI$, $C+kI$ and  $C'+kI$  are  invertible for all integers $k\geq0$. 
\end{corollary}
\begin{corollary}The following double integral representations:
\begin{align}&\Gamma_{\mathcal{A}}^{H}[(A; x), B, -mI; C+I, C'+I ; -z_1, z_2, z_3]\notag\\
&=m!\,z_{1}^{\frac{-C}{2}}\,[(C'+I)_{m}]^{-1}\, \Gamma^{-1}(A) \Gamma^{-1}(B)\Gamma(C+I)\notag\\
&\times  \int_{x}^{\infty}\int_{0}^{\infty} e^{-t-s} \,t^{A-\frac{C}{2}-I} s^{B-\frac{C}{2}-I} J_{C}(2\sqrt{z_1 st}) L^{(C')}_{m}(z_2 s+ z_3 t)dt ds;
\end{align}
\begin{align}&\Gamma_{\mathcal{A}}^{H}[(A; x), B, -mI; C+I, C'+I ;  z_1, z_2, z_3]\notag\\
&=m!\,z_{1}^{\frac{-C}{2}}\, [(C'+I)_{m}]^{-1}\,\Gamma^{-1}(A) \Gamma^{-1}(B)\Gamma(C+I)\notag\\
&\times  \int_{x}^{\infty}\int_{0}^{\infty} e^{-t-s} \,t^{A-\frac{C}{2}-I} s^{B-\frac{C}{2}-I} I_{C}(2\sqrt{z_1 st}) L^{(C')}_{m}(z_2 s+ z_3 t) dt ds,
\end{align}
where $A$, $B$, $B'$, $C$ and $C'$ are  positive stable and commutative matrices in  $\mathbb{C}^{r\times r}$ such that $A+kI$, $B+kI$, $C+kI$ and  $C'+kI$  are  invertible for all integers $k\geq0$. 
\end{corollary}
\begin{theorem}Let $A$, $B$, $B'$, $C$ and $C'$ be  positive stable and commutative matrices in  $\mathbb{C}^{r\times r}$ such that $A+kI$, $B+kI$, $B'+kI$, $C+kI$  and  $C'+kI$ are  invertible for all integers $k\geq0$. Then the following triple  integral representation for $\Gamma_{\mathcal{A}}^{H}$ in (\ref{2eq2}) holds true:
\begin{align}
&\Gamma_{\mathcal{A}}^{H}[(A; x), B, B'; C,C' ; z_1, z_2, z_3]\notag\\
&= \Gamma^{-1}(A) \Gamma^{-1}(B) \Gamma^{-1}(B')\Big[\int_{x}^{\infty}\int_{0}^{\infty} \int_{0}^{\infty}e^{-t-s-u}\, t^{A-I} s^{B-I}  u^{B'-I} \notag\\& \times {_{0}F_{1}}(-; C; z_1 st) \,{_{0}F_{1}}(-; C'; z_2 us+z_3ut) dt ds du\Big].\label{x1}
\end{align}
\end{theorem}
\begin{proof}By applying  the integral representation of the Pochhammer matrix symbol  $(B')_{n+p}$ in (\ref{l1}) and through applying the  summation formula (\ref{l2}) in tandem, one arrives at (\ref{x1}).
\end{proof}
\begin{corollary}The following triple integral representations holds true:
\begin{align}&\Gamma_{\mathcal{A}}^{H}[(A; x), B, B'; C+I, C' ; -z_1, z_2, z_3]\notag\\
&=z_{1}^{\frac{-C}{2}} \, \Gamma^{-1}(A) \Gamma^{-1}(B)\Gamma^{-1}(B')\Gamma(C+I)\notag\\
&\times  \int_{x}^{\infty}\int_{0}^{\infty}\int_{0}^{\infty} e^{-t-s-u} \,t^{A-\frac{C}{2}-I} s^{B-\frac{C}{2}-I} u^{B'-I}\notag\\&\times  J_{C}(2\sqrt{z_1 st}) \,{_{0}F_{1}}(-; C'; z_2 us+z_3ut)dt dsdu;
\end{align}
\begin{align}&\Gamma_{\mathcal{A}}^{H}[(A; x), B, B'; C+I, C' ; z_1, z_2, z_3]\notag\\
&=z_{1}^{\frac{-C}{2}} \, \Gamma^{-1}(A) \Gamma^{-1}(B)\Gamma^{-1}(B')\Gamma(C+I)\notag\\
&\times  \int_{x}^{\infty}\int_{0}^{\infty}\int_{0}^{\infty} e^{-t-s-u} \,t^{A-\frac{C}{2}-I} s^{B-\frac{C}{2}-I} u^{B'-I}\notag\\&\times  I_{C}(2\sqrt{z_1 st}) \,{_{0}F_{1}}(-; C'; z_2 us+z_3ut)dt dsdu,
\end{align}
where $A$, $B$, $B'$, $C$ and $C'$ are  positive stable and commutative matrices in  $\mathbb{C}^{r\times r}$ such that $A+kI$, $B+kI$, $C+kI$ and  $C'+kI$  are  invertible for all integers $k\geq0$. 
\end{corollary}
\begin{theorem}Let $A$, $B$, $B'$, $C$ and $C'$ be  positive stable and commutative  matrices in  $\mathbb{C}^{r\times r}$ such that $A+kI$, $B+kI$,  $C+kI$  and  $C'+kI$ are  invertible for all integers $k\geq0$. Then the  following reduction formula for $\Gamma_{\mathcal{A}}^{H}$  holds true:
\begin{align}
&\Gamma_{\mathcal{A}}^{H}[(A; x), B, B'; C,B' ; z_1, z_2, z_3]\notag\\
&= (1-z_3)^{-A} (1-z_2)^{-B}\notag\\
&\times {_{2}\Gamma_{1}}[(A; x(1-z_3)), B; C; \frac{z_1}{(1-z_2)(1-z_3)}].\label{x1i}
\end{align}
\end{theorem}
\begin{proof}Replacing  $B'=C'$ in (\ref{2eq6}),  the contrived  equation becomes
\begin{align}
&\Gamma_{\mathcal{A}}^{H}[(A; x), B, B'; C,B'; z_1, z_2, z_3]\notag\\&= \Gamma^{-1}(A) \Gamma^{-1}(B) \Big[\int_{x}^{\infty} \int_{0}^{\infty}e^{-t(1-z_3)-s(1-z_2)}\, t^{A-I} s^{B-I}  {_{0}F_{1}}(-; C; z_1 st)\,  dt ds\Big],\label{ii1}
\end{align}
where $F(-; -;z)= e^{z}$. Putting $t(1-z_3)= u,$  $s(1-z_2)=v$ and $dt= \frac{du}{(1-z_3)}$, $ds= \frac{dv}{(1-z_2)}$ in (\ref{ii1}),  the following simplified equation emerges:
\begin{align}
&\Gamma_{\mathcal{A}}^{H}[(A; x), B, B'; C,B'; z_1, z_2, z_3]\notag\\& =  \Gamma^{-1}(A) (1-z_3)^{-A}(1-z_2)^{-B}\notag\\ &\times \Big[\int_{x(1-z_3)}^{\infty}e^{-u}\, u^{A-I}   {_{1}F_{1}}(B; C; \frac{z_1u}{(1-z_2)(1-z_3)})\,  du\Big].\label{ii2}
\end{align}
Finally, using known result in \cite{Ab}:
\begin{align}
{_{2}\Gamma_{1}}[[A; x], B;C;z]= \Gamma^{-1}(A) \int_{x}^{\infty} e^{-t} t^{A-I} {_{1}F_{1}}(B; C; zt) dt
\end{align}
in (\ref{ii2}), we are led to the desired result (\ref{x1i}).
\end{proof}
\begin{theorem}\label{h1}Let $B'+sI$  be  invertible  for all integers $s\geq0$. Then the following recursion formula holds  true for  $\Gamma_{\mathcal{A}}^{H}$:
\begin{align}
&\Gamma_{\mathcal{A}}^{H}[(A; x), B, B'+sI; C,C' ; z_1, z_2, z_3]\notag\\
&=\Gamma_{\mathcal{A}}^{H}[(A; x), B, B'; C, C' ; z_1, z_2, z_3]\notag\\&\quad+ z_2{B}{C'}^{-1}\Big[\sum_{k=1}^{s}\Gamma_{\mathcal{A}}^{H}[(A; x), B+I, B'+kI; C, C'+I ; z_1, z_2, z_3]\Big]\notag\\
&\quad+ z_3{A}{C'}^{-1}\Big[\sum_{k=1}^{s}\Gamma_{\mathcal{A}}^{H}[(A+I; x), B, B'+kI; C,  C'+I; z_1, z_2, z_3]\Big].\label{r11}
\end{align}
Furthermore, if  $B'-kI$ are invertible for integers $k\leq s$, then 
\begin{align}&\Gamma_{\mathcal{A}}^{H}[(A; x), B, B'-sI; C,C' ;z_1, z_2, z_3]\notag\\
&=\Gamma_{\mathcal{A}}^{H}[(A; x), B, B'; C,C' ;z_1, z_2, z_3]\notag\\&\quad- z_2{B}{C'}^{-1}\Big[\sum_{k=0}^{s-1}\Gamma_{\mathcal{A}}^{H}[(A; x), B+I, B'-kI; C, C'+I ; z_1, z_2, z_3]\Big]\notag\\
&\quad- z_3{A}{C'}^{-1}\Big[\sum_{k=0}^{s-1}\Gamma_{\mathcal{A}}^{H}[(A+I; x), B, B'-kI; C, C'+I; z_1, z_2, z_3]\Big].\label{r12}
\end{align}
where $A$, $B$, $B'$, $C$  and $C'$ are  positive stable and commutative matrices in $\mathbb{C}^{r\times r}$ such that $C+kI$ and $C'+kI$  are  invertible  matrices for all  integers $k\geq 0$.
\end{theorem}
\begin{proof}
 Using  integral formula  (\ref{2eq6}) of the incomplete Srivastava's  triple hypergeometric  matrix function $\Gamma_{\mathcal{A}}^{H}$  and the transformation  of the said equation
\begin{align*}
(B'+I)_{n+p}= B'^{-1}(B')_{n+p}(B'+(n+p)I),
\end{align*}
we get the following contiguous  matrix relation:
\begin{align}
&\Gamma_{\mathcal{A}}^{H}[(A; x), B, B'+I; C, C' ;z_1, z_2, z_3]\notag\\
&=\Gamma_{\mathcal{A}}^{H}[(A; x), B, B'; C, C' ; z_1, z_2, z_3]\notag\\&\quad+ z_2{B}{C'}^{-1}\Big[\Gamma_{\mathcal{A}}^{H}[(A; x), B+I, B'+I; C, C'+I ; z_1, z_2, z_3]\Big]\notag\\
&\quad+ z_3{A}{C'}^{-1}\Big[\Gamma_{\mathcal{A}}^{H}[(A+I; x), B, B'+I; C, C'+I; z_1, z_2, z_3]\Big].
\label{3eqp1}
\end{align}  
Application of  this contiguous matrix relation to the  matrix function $\Gamma_{\mathcal{A}}^{H}$  with the matrix parameter $B'+2I$, yields

\begin{align}
&\Gamma_{\mathcal{A}}^{H}[(A; x), B, B'+2I; C,C'; z_1, z_2, z_3]\notag\\
&=\Gamma_{\mathcal{A}}^{H}[(A; x), B, B'; C,C' ; z_1, z_2, z_3]\notag\\&\quad+ z_2{B}{C'}^{-1}\Big[\sum_{k=1}^{2}\Gamma_{\mathcal{A}}^{H}[(A; x), B+I, B'+kI; C, C'+I ; z_1, z_2, z_3]\Big]\notag\\
&\quad+ z_3{A}{C'}^{-1}\Big[\sum_{k=1}^{2}\Gamma_{\mathcal{A}}^{H}[(A+I; x), B, B'+kI; C,  C'+I; z_1, z_2, z_3]\Big].
\label{3eqp2}
\end{align}  
Repeating this process $s$ times,  we obtain (\ref{r11}). 

 For  the proof of (\ref{r12}),  replace the matrix  $B'$ with $B'-I$ in (\ref{3eqp1}). As $B'-I$ is invertible, this gives 
\begin{align}
&\Gamma_{\mathcal{A}}^{H}[(A; x), B, B'- I; C,C'; z_1, z_2, z_3]\notag\\
&=\Gamma_{\mathcal{A}}^{H}[(A; x), B, B'; C,C'; z_1, z_2, z_3]\notag\\&\quad- z_2{B}{C'}^{-1}\Big[\Gamma_{\mathcal{A}}^{H}[(A; x), B+I, B'; C, C'+I ; z_1, z_2, z_3]\Big]\notag\\
&\quad- z_3{A}{C'}^{-1}\Big[\Gamma_{\mathcal{A}}^{H}[(A+I; x), B, B'; C, C'+I; z_1, z_2, z_3]\Big].\label{o1}
\end{align}
Iteratively, we get (\ref{r12}).  This finishes the proof of  theorem (\ref{h1}).
\end{proof}
Other recursion formulas for the matrix functions $\Gamma_{\mathcal{A}}^{H}$ about the matrix parameter  $B'$ can be obtained as follows:
\begin{theorem}Let $B'+sI$  be  invertible  for all  integers $s\geq0$. Then the following recursion formula holds  true for  $\Gamma_{\mathcal{A}}^{H}$:
\begin{align}
&\Gamma_{\mathcal{A}}^{H}[(A; x), B, B'+sI; C, C' ; z_1, z_2, z_3]\notag\\
&=\sum_{k_1+k_2\leq s}^{}{s\choose k_1, k_2}z_{2}^{k_1} z_{3}^{k_2}{(A)_{k_2}}(B)_{k_1}(C')^{-1}_{k_1+k_2} \,\notag\\
&\quad\times\Big[\Gamma_{\mathcal{A}}^{H}[(A+k_2 I; x), B+k_1I, B'+(k_1+k_2)I; C,  C'+(k_1+k_2)I; z_1, z_2, z_3]\Big].\label{3eqh1}
\end{align}
Furthermore, if  $B'-kI$ are invertible for  integers $k\leq s$, then 
\begin{align}&\Gamma_{\mathcal{A}}^{H}[(A; x), B, B'-sI; C, C'; z_1, z_2, z_3]\notag\\
&=\sum_{k_1+k_2\leq s}^{}{s\choose k_1, k_2}(-z_{2})^{k_1} (-z_{3})^{k_2}{(A)_{k_2}}(B)_{k_1}(C')^{-1}_{k_1+k_2}\,\notag\\
&\quad\times\Big[\Gamma_{\mathcal{A}}^{H}[(A+k_2 I; x), B+k_1I, B'; C, C'+(k_1+k_2 )I; z_1, z_2, z_3]\Big],\label{3eqh2}
\end{align}
where $A$, $B$, $B'$, $C$ and $C'$ are positive stable and commutative matrices in $\mathbb{C}^{r\times r}$ such that $C+kI$  and $C'+kI$  are invertible  matrices for all integers $k\geq 0$.
\end{theorem}
\begin{proof}
The proof of (\ref{3eqh1}) is based upon the principle of mathematical induction on  $s\in\mathbb{N}$. For $s=1$, the result (\ref{3eqh1}) is verifiably  true. Suppose  (\ref{3eqh1}) is true for $s=t$, that is,
 \begin{align}
&\Gamma_{\mathcal{A}}^{H}[(A; x), B, B'+tI; C, C' ; z_1, z_2, z_3]\notag\\
&=\sum_{k_1+k_2\leq t}^{}{t\choose k_1, k_2}z_{2}^{k_1} z_{3}^{k_2}{(A)_{k_2}}(B)_{k_1}\,(C')^{-1}_{k_1+k_2}\,\notag\\
&\quad\times\Big[\Gamma_{\mathcal{A}}^{H}[(A+k_2 I; x), B+k_1I, B'+(k_1+k_2)I; C, C'+(k_1+k_2) I; z_1, z_2, z_3]\Big],
 \label{3eqp3}
 \end{align}
 Replacing $B'$ with $B'+I$ in (\ref{3eqp3}) and using the contiguous matrix relation (\ref{3eqp1}),  the equation becomes
 \begin{align}
 &\Gamma_{\mathcal{A}}^{H}[(A; x), B, B'+(t+1)I; C, C' ; z_1, z_2, z_3]\notag\\
 &=\sum_{k_1+k_2\leq t}^{}{t\choose k_1, k_2}z_{2}^{k_1} z_{3}^{k_2}{(A)_{k_2}}(B)_{k_1}\,(C')^{-1}_{k_1+k_2}\,\notag\\
 &\quad\times\Big[\Gamma_{\mathcal{A}}^{H}[(A+k_2 I; x), B+k_1I, B'+(k_1+k_2)I; C, C'+(k_1+k_2) I; z_1, z_2, z_3]\notag\\&\quad+z_2{(B+k_1I)}(C'+(k_1+k_2)I)^{-1}\notag\\&\quad\times \Gamma_{\mathcal{A}}^{H}[(A+k_2I;x), B+k_1I+I, B'+(k_1+k_2+1)I; C, C'+(k_1+k_2+1)I ; z_1, z_2, z_3]\notag\\
 &\quad+z_3 {(A+k_2I)}(C'+(k_1+k_2)I)^{-1}\notag\\&\quad \times \Gamma_{\mathcal{A}}^{H}[(A+(k_2+1)I;x), B+k_1I, B'+(k_1+k_2+1)I; C, C'+(k_1+k_2+1)I; z_1, z_2, z_3]\Big].
 \label{3eqp4}
 \end{align}
 Simplifying, (\ref{3eqp4}) takes the form
 \begin{align}
 &\Gamma_{\mathcal{A}}^{H}[(A; x), B, B'+(t+1)I; C, C' ; z_1, z_2, z_3]\notag\\
 &=\sum_{k_1+k_2\leq t}^{}{t\choose k_1, k_2}z_{2}^{k_1} z_{3}^{k_2}{(A)_{k_2}}(B)_{k_1}\,(C')^{-1}_{k_1+k_2}\,\notag\\
 &\quad\times\Gamma_{\mathcal{A}}^{H}[(A+k_2 I; x), B+k_1I, B'+(k_1+k_2)I; C, C'+(k_1+k_2) I; z_1, z_2, z_3]\notag\\
 &\quad +\sum_{k_1+k_2\leq t+1}^{}{t\choose k_1-1, k_2}z_{2}^{k_1} z_{3}^{k_2}{(A)_{k_2}}(B)_{k_1}\,(C')^{-1}_{k_1+k_2}\notag\\&\quad\times\Gamma_{\mathcal{A}}^{H}[(A+k_2 I; x), B+k_1I, B'+(k_1+k_2)I; C, C'+(k_1+k_2) I; z_1, z_2, z_3]\notag\\
 &\quad +\sum_{k_1+k_2\leq t+1}^{}{t\choose k_1, k_2-1}z_{2}^{k_1} z_{3}^{k_2}{(A)_{k_2}}(B)_{k_1}\,(C')^{-1}_{k_1+k_2}\notag\\&\times\Gamma_{\mathcal{A}}^{H}[(A+k_2 I; x), B+k_1I, B'+(k_1+k_2)I; C, C'+(k_1+k_2) I; z_1, z_2, z_3].
 \label{3eqp5}
 \end{align}
 Using Pascal's identity in (\ref{3eqp5}), we have
 \begin{align}
& \Gamma_{\mathcal{A}}^{H}[(A; x), B, B'+(t+1)I; C, C'; z_1, z_2, z_3]\notag\\
&=\sum_{k_1+k_2\leq t+1}^{}{t+1\choose k_1, k_2}z_{2}^{k_1} z_{3}^{k_2}{(A)_{k_2}}(B)_{k_1}\,(C')^{-1}_{k_1+k_2}\,\notag\\
 &\quad\times\Gamma_{\mathcal{A}}^{H}[(A+k_2 I; x), B+k_1I, B'+(k_1+k_2)I; C, C'+(k_1+k_2) I; z_1, z_2, z_3].
 \end{align}
 This establishes (\ref{3eqh1}) for $s=t+1$. Hence through  induction the result given in (\ref{3eqh1}) stands true for all values of $s$. The second recursion formula (\ref{3eqh2}) can be proved in a similar manner.\end{proof}
\begin{theorem}
Let $C-sI$  and $C'-sI$ be invertible matrices for all integers $s\geq0$. Then the following recursion formulas hold  true for  $\Gamma_{\mathcal{A}}^{H}$:
\begin{align}
&\Gamma_{\mathcal{A}}^{H}[(A; x), B, B'; C-sI, C'; z_1, z_2, z_3]\notag\\
&=\Gamma_{\mathcal{A}}^{H}[(A; x), B, B'; C, C' ; z_1, z_2, z_3]\notag\\&\quad+ z_1 AB\Big[\sum_{k=1}^{s} \Gamma_{\mathcal{A}}^{H}[(A+I; x), B+I, B'; C+(2-k)I, C' ; z_1, z_2, z_3]\notag\\&\quad\times{(C-kI)^{-1}(C-(k-1)I)^{-1}}\Big];\label{zz1}
\end{align}
\begin{align}
&\Gamma_{\mathcal{A}}^{H}[(A; x), B, B'; C, C'-sI ; z_1, z_2, z_3]\notag\\
&=\Gamma_{\mathcal{A}}^{H}[(A; x), B, B'; C, C'; z_1, z_2, z_3]\notag\\&\quad+ z_2 B B'\Big[\sum_{k=1}^{s} \Gamma_{\mathcal{A}}^{H}[(A; x), B+I, B'+I;\,C, C'+(2-k)I ; z_1, z_2, z_3]\,\notag\\
&\quad\times{(C'-kI)^{-1}(C'-(k-1)I)^{-1}}\Big]\notag\\
&+ z_3 A B'\Big[\sum_{k=1}^{s} \Gamma_{\mathcal{A}}^{H}[(A+I; x), B, B'+I;\,C, C'+(2-k)I ; z_1, z_2, z_3]\notag\\
&\quad\times{(C'-kI)^{-1}(C'-(k-1)I)^{-1}}\Big],\label{ph2}
\end{align}
where $A$, $B$, $B'$, $C$ and $C'$ are positive stable and commutative matrices in $\mathbb{C}^{r\times r}$ such that $C+kI$ and $C'+kI$  are invertible  matrices for all integers $k\geq 0$.
\end{theorem}
\begin{proof}
 Applying   integral formula  (\ref{2eq6}) of the incomplete  Srivastava's  triple hypergeometric  matrix function $\Gamma_{\mathcal{A}}^{H}$  and the transformation which follows
\begin{align*}
(C-I)^{-1}_{m}=(C)^{-1}_{m}\left[I+{m}{(C-I)^{-1}}\right],
\end{align*}
we can easily get the contiguous matrix relation
\begin{align}
&\Gamma_{\mathcal{A}}^{H}[(A; x), B, B'; C-I, C'; z_1, z_2, z_3]\notag\\
&=\Gamma_{\mathcal{A}}^{H}[(A; x), B, B'; C, C' ; z_1, z_2, z_3]\notag\\&\quad+ z_1 AB \,\Big[\Gamma_{\mathcal{A}}^{H}[(A+I; x), B+I, B'; C+I, C' ; z_1, z_2, z_3]\Big]{C^{-1}(C-I)^{-1}}.
\label{2eqp9}
\end{align}
Applying this contiguous matrix relation twice on the  incomplete  Srivastava's  triple hypergeometric  matrix function $\Gamma_{\mathcal{B}}^{H}$ with matrix $C-2I$, we arrive at,
\begin{align}
&\Gamma_{\mathcal{A}}^{H}[(A; x), B, B'; C-2I, C' ; z_1, z_2, z_3]\notag\\
&=\Gamma_{\mathcal{A}}^{H}[(A; x), B, B'; C-I, C' ; z_1, z_2, z_3]\notag\\&\quad+ z_1 AB \,\Big[\Gamma_{\mathcal{A}}^{H}[(A+I; x), B+I, B'; C, C' ; z_1, z_2, z_3]\Big]{(C-I)^{-1}(C-2I)^{-1}}\notag\\
&=\Gamma_{\mathcal{A}}^{H}[(A; x), B, B'; C; z_1, z_2, z_3]\notag\\&\quad+ z_1 AB \Big[{\Gamma_{\mathcal{A}}^{H}[(A+I;x), B+I, B'; C+I, C'; z_1, z_2, z_3]}{C^{-1}(C-I)^{-1}}\notag\\
&\quad +\,{\Gamma_{\mathcal{A}}^{H}[(A+I; x), B+I, B'; C, C' ; z_1, z_2, z_3]}{(C-I)^{-1}(C-2I)^{-1}}\Big].
\label{2eqp19}
\end{align}
Iterating this method $s$ times on the  incomplete Srivastava's  triple hypergeometric  matrix function  $\Gamma_{\mathcal{A}}^{H}((A; x), B, B'; C-sI, C'; x_1, x_2, x_3)$  we get the recursion formula (\ref{zz1}). Recursion formulas (\ref{ph2})  can be proved in an analogous manner.\end{proof}
\begin{theorem}\label{r1} Let $A$, $B$,  $B'$, $C$ and $C'$  be  positive stable and commutative  matrices in  $\mathbb{C}^{r\times r}$ such that $C+kI$ and $C'+kI$ are invertible  matrices for all integers $k\geq 0$. Then the  following recurrence relations hold true:
\begin{align}
&(C'-(B'+I))\Gamma_{\mathcal{A}}^{H}[(A; x), B, B'; C, C'; z_1, z_2, z_3]\notag\\
&= (C'-I) \Gamma_{\mathcal{A}}^{H}[(A; x), B, B'; C, C'-I; z_1, z_2, z_3]\notag\\
&\quad- B' \,\Gamma_{\mathcal{A}}^{H}[(A; x), B, B'+I; C, C'; z_1, z_2, z_3].\label{2eq15}\end{align}

\end{theorem}
\begin{proof} Using the integral formula (\ref{2eq6}) and the following contiguous matrix  relation \cite{ZM}
\begin{align}&(C'-(B'+I)) _{1}F_{1}(B'; C'; z)\notag\\&=(C'-I) _{1}F_{1}(B'; C'-I; z)- B'\,{ _{1}F_{1}}(B'+I; C'; z), 
\end{align}
The following recurrence relation emerges:
\begin{align}
&(C'-(B'+I))\Gamma_{\mathcal{A}}^{H}[(A; x), B, B'; C,C'; z_1, z_2, z_3]\notag\\&= \Gamma^{-1}(A)\Gamma^{-1}(B)\Big[\int_{x}^{\infty} e^{-t-s}t^{A-I}s^{B-I} \Big((C'-I) {_{1}F_{1}}[B'; C'-I; z_2s+z_3t]\notag\\&\quad -B' {_{1}F_{1}}(B'+I; C'; z_2s+z_3t)\Big)  {_{0}F_{1}}(-; C; z_1st) dtds\Big],\notag\\
&= (C'-I) \Gamma_{\mathcal{A}}^{H}[(A; x), B, B'; C, C'-I; z_1, z_2, z_3]- B' \,\Gamma_{\mathcal{A}}^{H}[(A; x), B, B'+I; C, C'; z_1, z_2, z_3].
\end{align}
This finishes the proof of (\ref{2eq15}) as well as completes the proof of Theorem~\ref{r1}.
\end{proof}
\begin{theorem} Let  $A$, $B$, $B'$, $C$  and $C'$ be  positive stable  and commutative matrices  in  $\mathbb{C}^{r\times r}$ such that  $C+kI$  and  $C'+kI$ are invertible for all integers $k\geq0$. Then the following  recurrence relation  for $\Gamma_{\mathcal{A}}^{H}$ in (\ref{2eq2}) holds true: 
\begin{align}&\Gamma_{\mathcal{A}}^{H}[(A; x), B, B'; C, C' ; z_1, z_2, z_3]=\Gamma_{\mathcal{A}}^{H}[(A; x), B, B'; C-I, C' ; z_1, z_2, z_3]\notag\\
&-ABC^{-1}(C-I)^{-1} \Gamma_{\mathcal{A}}^{H}[(A+I; x), B+I, B'; C+I,  C'; z_1, z_2, z_3].
\end{align}
\end{theorem}
\begin{proof}Applying  the contiguous relation for the function $_{0}F_{1}$
\begin{align}{_{0}F_{1}}(-; C-I; z_1)-  {_{0}F_{1}}(-; C; z_1)- z_1 C^{-1}(C-I)^{-1} \,{_{0}F_{1}}(-; C+I; z_1)=0
\end{align}
in the integral representation (\ref{2eq6}),  we are led to the desired result.
\end{proof}
\begin{theorem}Let $A$, $B$, $B'$, $C$  and $C'$ be  positive stable  and commutative matrices  in  $\mathbb{C}^{r\times r}$ such that  $C+kI$  and  $C'+kI$  are invertible for all integers $k\geq0$. Then the following  derivative formulas  for $\Gamma_{\mathcal{A}}^{H}$ in (\ref{2eq2}) hold true: 
\begin{align}
&\frac{\partial^{m}}{\partial z_{1}^{m}}\Gamma_{\mathcal{A}}^{H}[(A; x), B, B'; C, C' ; z_1, z_2, z_3]\notag\\&= (A)_{m}(B)_{m} (C)^{-1}_{m}\Big[\Gamma_{\mathcal{A}}^{H}[(A+mI; x), B+mI, B'; C+mI, C' ; z_1, z_2, z_3]\Big];\label{d1}\\
&\frac{\partial^{m+n}}{\partial z_{1}^{m}\partial z_{2}^{n}}\Gamma_{\mathcal{A}}^{H}[(A; x), B, B'; C, C'; z_1, z_2, z_3]\notag\\&= (A)_{m}(B)_{m+n}(B')_{n}(C)^{-1}_{m} (C')^{-1}_{n}\notag\\& \times \Big[\Gamma_{\mathcal{A}}^{H}[(A+mI; x), B+(m+n)I, B'+nI; C+mI,C'+nI ; z_1, z_2, z_3]\Big];\label{d2}\\
&\frac{\partial^{m+n+p}}{\partial z_{1}^{m}\partial z_{2}^{n}\partial z_{3}^{p}}\Gamma_{\mathcal{A}}^{H}[(A; x), B, B'; C,C' ; z_1, z_2, z_3]\notag\\&= (A)_{m+p}(B)_{m+n}(B')_{n+p}(C)^{-1}_{m} (C')^{-1}_{n+p}\notag\\& \times \Big[\Gamma_{\mathcal{A}}^{H}[(A+(m+p)I; x), B+(m+n)I, B'+(n+p)I; C+mI,C'+(n+p)I; z_1, z_2, z_3]\Big].\label{d3}
\end{align}
\end{theorem}
\begin{proof}Upon differentiating formally both sides of (\ref{2eq6}) with respect to $z_1$, the resultant equation comes to:
\begin{align}&\frac{\partial}{\partial z_{1}}\Gamma_{\mathcal{A}}^{H}[(A; x), B, B'; C, C' ; z_1, z_2, z_3]
=AB C^{-1}\Gamma^{-1}(A+I) \Gamma^{-1}(B+I)\notag\\&\times \Big[\int_{x}^{\infty} \int_{0}^{\infty}e^{-t-s}\, t^{(A+I)-I} s^{(B+I)-I}  {_{0}F_{1}}(-; C+I; z_1 st) {_1F_{1}}(B'; C' ; z_2 s+ z_3 t) dt ds\Big].\label{d4}
\end{align}
Now,  from  (\ref{2eq6}) and (\ref{d4}), we obtain
\begin{align}&\frac{\partial}{\partial z_{1}}\Gamma_{\mathcal{A}}^{H}[(A; x), B, B'; C, C' ; z_1, z_2, z_3]\notag\\&= AB C^{-1}\Big[\Gamma_{\mathcal{A}}^{H}[(A+I; x), B+I, B'; C+I, C' ; z_1, z_2, z_3]\Big],
\end{align}
which is (\ref{d1}) for $m=1$. The general result follows by the principle of mathematical induction on $m$.
\end{proof}

This completes the proof of (\ref{d1}). Successively  (\ref{d2}) and (\ref{d3}) can be proved in an analogous manner.

\section{The  incomplete Srivastava's  triple  hypergeometric matrix  functions $\gamma_{\mathcal{B}}^{H}$ and $\Gamma_{\mathcal{B}}^{H}$}
In this section, we introduce the incomplete Srivastava's  triple  hypergeometric matrix  functions $\gamma_{\mathcal{B}}^{H}$ and $\Gamma_{\mathcal{B}}^{H}$ as follows: 
\begin{align}
&\gamma_{\mathcal{B}}^{H}[(A; x), B, B'; C,C', C''; z_1, z_2, z_3]\notag\\&=\sum_{m, n, p\geq 0}(A;x)_{m+p} (B)_{m+n} (B')_{n+p} (C)^{-1}_{m}(C')^{-1}_{n}(C'')^{-1}_{p}\frac{z_{1}^{m}z_2^{n}z_3^{p}}{m! n! p!},\label{s2eq1}\\
&\Gamma_{\mathcal{B}}^{H}[(A; x), B, B'; C,C', C''; z_1, z_2, z_3]\notag\\&=\sum_{m, n, p\geq 0}[A;x]_{m+p} (B)_{m+n} (B')_{n+p} (C)^{-1}_{m}(C')^{-1}_{n}(C'')^{-1}_{p}\frac{z_{1}^{m}z_2^{n}z_3^{p}}{m! n! p!}.\label{s2eq2}
\end{align}
where $A$, $B$, $B'$, $C$, $C'$  and $C''$ are  positive stable and commutative matrices in  $\mathbb{C}^{r\times r}$ such that $C+kI$, $C'+kI$ and  $C''+kI$  are invertible for all integers $k\geq0$.\\
From (\ref{1eq9}), we have the following decomposition formula
\begin{align}
&\gamma_{\mathcal{B}}^{H}[(A; x), B, B'; C,C', C''; z_1, z_2, z_3]+ \Gamma_{\mathcal{B}}^{H}[(A; x), B, B'; C,C', C''; z_1, z_2, z_3]\notag\\&= H_{\mathcal{B}}[A, B, B'; C,C', C''; z_1, z_2, z_3].\label{s2eq3}
\end{align}
where $H_{\mathcal{B}}[A, B, B'; C,C', C''; z_1, z_2, z_3]$ is the Srivastava's  triple  hypergeometric matrix  functions \cite{RD1}.

\begin{remark}
It is quite evident to note that the special cases of (\ref{s2eq1}) and (\ref{s2eq2}) when $z_2 = 0$ reduce to the
known incomplete families of the second Appell hypergeometric  matrix functions \cite{A3}. Also, the special
cases of (\ref{s2eq1}) and (\ref{s2eq2}) when $z_2 = 0$ and  $z_3 = 0$ or $z_1 = 0$ are seen to yield the known incomplete families of Gauss hypergeometric  matrix functions \cite{Ab}.\end{remark}

To discuss the properties and characteristics of $\gamma_{\mathcal{B}}^{H}[(A; x), B, B'; C,C', C''; z_1, z_2, z_3]$, it is sufficient to determine the properties of $\Gamma_{\mathcal{B}}^{H}[(A; x), B, B'; C,C', C''; z_1, z_2, z_3]$ according to the decomposition formula (\ref{s2eq3}). We omit the proofs of the given below theorems.
\begin{theorem}Let $A$, $B$, $B'$, $C$, $C'$  and $C''$ be  positive stable and commutative matrices in  $\mathbb{C}^{r\times r}$ such that $C+kI$, $C'+kI$ and $C''+kI$ is invertible for all integer $k\geq0.$ Then the function defined by 
${\mathcal{T}_2}={\mathcal{T}_2}(z_1, z_2, z_3) =\gamma_{\mathcal{B}}^{H}[(A; x), B, B'; C,C', C''; z_1, z_2, z_3]+\Gamma_{\mathcal{B}}^{H}[(A; x), B, B'; C,C', C''; z_1, z_2, z_3]$   satisfies the following system of partial differential equations:
\begin{align}
&\Big[z_1\frac{\partial}{\partial z_1}{(z_1\frac{\partial}{\partial z_1}+C-I)- z_1(z_1\frac{\partial}{\partial z_1}+z_3\frac{\partial}{\partial z_3}+A)(z_1\frac{\partial}{\partial z_1}+z_2\frac{\partial}{\partial z_2}+B)}\Big] \mathcal{T}=O,\label{2eq4}\\
&\Big[z_2\frac{\partial}{\partial z_2}{(z_2\frac{\partial}{\partial z_2}+C'-I)- z_2(z_2\frac{\partial}{\partial z_2}+z_1\frac{\partial}{\partial z_1}+B)(z_2\frac{\partial}{\partial z_2}+z_3\frac{\partial}{\partial z_3}+B')}\Big] \mathcal{T}=O,\label{m1}\\
&\Big[z_3\frac{\partial}{\partial z_3}{(z_3\frac{\partial}{\partial z_3}+C''-I)- z_3(z_1\frac{\partial}{\partial z_1}+z_3\frac{\partial}{\partial z_3}+A)(z_2\frac{\partial}{\partial z_2}+z_3\frac{\partial}{\partial z_3}+B')}\Big] \mathcal{T}=O.\label{s2eq5}
\end{align}
\end{theorem}
\begin{theorem} Let $A$, $B$, $B'$, $C$, $C'$ and $C''$ be  positive stable and commutative  matrices in  $\mathbb{C}^{r\times r}$ such that $A+kI$, $B+kI$, $C+kI$, $C'+kI$ and  $C''+kI$  are  invertible for all integers $k\geq0$. Then the following integral representation for $\Gamma_{\mathcal{B}}^{H}$ in (\ref{s2eq2}) holds true:
\begin{align}
&\Gamma_{\mathcal{B}}^{H}[(A; x), B, B'; C,C', C''; z_1, z_2, z_3]= \Gamma^{-1}(A) \Gamma^{-1}(B)\notag\\&\times \Big[\int_{x}^{\infty} \int_{0}^{\infty}e^{-t-s}\, t^{A-I} s^{B-I}  {_{0}F_{1}}(-; C; z_1 st) \Psi_{2}(B'; C', C''; z_2 s, z_3 t) dt ds\Big],\label{s2eq6}
\end{align}where $\Psi_{2}$ is the Humbert's matrix function  defined by \cite{MA, SZ}
\begin{align}\Psi_{2}(A; C, C'; z_1, z_2)=\sum_{m, n \geq 0}^{}(A)_{m+n} (C)^{-1}_{m} (C')^{-1}_{n}\frac{z_{1}^{m} z_{2}^{n}}{m! n!}.
\end{align} 
\end{theorem}
\begin{corollary}The following double integral representations:
\begin{align}&\Gamma_{\mathcal{B}}^{H}[(A; x), B, B'; C+I,C', C''; -z_1, z_2, z_3]\notag\\
&=x_{1}^{\frac{-C}{2}}\, \Gamma^{-1}(A) \Gamma^{-1}(B)\Gamma(C+I)\notag\\
&\times  \int_{x}^{\infty}\int_{0}^{\infty} e^{-t-s} \,t^{A-\frac{C}{2}-I} s^{B-\frac{C}{2}-I} J_{C}(2\sqrt{z_1 st}) \Psi_{2}(B'; C', C''; z_2 s, z_3 t) dt ds
\end{align}
and
\begin{align}&\Gamma_{\mathcal{B}}^{H}[(A; x), B, B'; C+I,C', C'';  z_1, z_2, z_3]\notag\\
&=x_{1}^{\frac{-C}{2}}\, \Gamma^{-1}(A) \Gamma^{-1}(B)\Gamma(C+I)\notag\\
&\times  \int_{x}^{\infty}\int_{0}^{\infty} e^{-t-s} \,t^{A-\frac{C}{2}-I} s^{B-\frac{C}{2}-I} I_{C}(2\sqrt{z_1 st}) \Psi_{2}(B'; C', C''; z_2 s, z_3 t) dt ds,
\end{align}
where $A$, $B$, $B'$, $C$, $C'$ and $C''$ are  positive stable and commutative  matrices in  $\mathbb{C}^{r\times r}$ such that $A+kI$, $B+kI$, $C+kI$, $C'+kI$ and  $C''+kI$  are  invertible for all integers $k\geq0$.
\end{corollary}
\begin{theorem}Let $A$, $B$, $B'$, $C$, $C'$ and $C''$ be  positive stable and commutative  matrices in  $\mathbb{C}^{r\times r}$ such that $A+kI$, $B+kI$, $B'+kI$, $C+kI$, $C'+kI$ and  $C''+kI$ are  invertible for all integers $k\geq0$. Then the following triple  integral representation for $\Gamma_{\mathcal{B}}^{H}$ in (\ref{s2eq2}) holds true:
\begin{align}
&\Gamma_{\mathcal{B}}^{H}[(A; x), B, B'; C,C', C''; z_1, z_2, z_3]\notag\\
&= \Gamma^{-1}(A) \Gamma^{-1}(B) \Gamma^{-1}(B')\Big[\int_{x}^{\infty}\int_{0}^{\infty} \int_{0}^{\infty}e^{-t-s-u}\, t^{A-I} s^{B-I}  u^{B'-I} \notag\\& \times {_{0}F_{1}}(-; C; z_1 st) \,{_{0}F_{1}}(-; C'; z_2 us)  \,{_{0}F_{1}}(-; C''; z_3 ut) dt ds du\Big].\label{sx1}
\end{align}
\end{theorem}
\begin{corollary}The following triple integral representations holds true:
\begin{align}&\Gamma_{\mathcal{B}}^{H}[(A; x), B, B'; C+I, C'+I, C''+I; -z_1, -z_2, -z_3]\notag\\
&=z_{1}^{\frac{-C}{2}} z_{2}^{\frac{-C'}{2}} z_{3}^{\frac{-C''}{2}}\, \Gamma^{-1}(A) \Gamma^{-1}(B)\Gamma^{-1}(B')\Gamma(C+I)\Gamma(C'+I)\Gamma(C''+I)\notag\\
&\times  \int_{x}^{\infty}\int_{0}^{\infty}\int_{0}^{\infty} e^{-t-s-u} \,t^{A-\frac{C}{2}-\frac{C''}{2}-I} s^{B-\frac{C}{2}-\frac{C'}{2}-I} u^{B'-\frac{C'}{2}-\frac{C''}{2}-I}\notag\\&\times  J_{C}(2\sqrt{z_1 st}) J_{C'}(2\sqrt{z_2 su})J_{C''}(2\sqrt{z_3 ut})dt dsdu
\end{align}
and
\begin{align}&\Gamma_{\mathcal{B}}^{H}[(A; x), B, B'; C+I, C'+I, C''+I; z_1, z_2, z_3]\notag\\
&=z_{1}^{\frac{-C}{2}} z_{2}^{\frac{-C'}{2}} z_{3}^{\frac{-C''}{2}}\, \Gamma^{-1}(A) \Gamma^{-1}(B)\Gamma^{-1}(B')\Gamma(C+I)\Gamma(C'+I)\Gamma(C''+I)\notag\\
&\times  \int_{x}^{\infty}\int_{0}^{\infty}\int_{0}^{\infty} e^{-t-s-u} \,t^{A-\frac{C}{2}-\frac{C''}{2}-I} s^{B-\frac{C}{2}-\frac{C'}{2}-I} u^{B'-\frac{C'}{2}-\frac{C''}{2}-I}\notag\\&\times  I_{C}(2\sqrt{z_1 st}) I_{C'}(2\sqrt{z_2 su})I_{C''}(2\sqrt{z_3 ut})dt dsdu,
\end{align}
where $A$, $B$, $B'$, $C$, $C'$ and $C''$ be  positive stable and commutative  matrices in  $\mathbb{C}^{r\times r}$ such that $A+kI$, $B+kI$, $B'+kI$, $C+kI$, $C'+kI$ and  $C''+kI$ are  invertible for all integers $k\geq0$.
\end{corollary}
\begin{theorem}\label{h1}Let $B'+sI$  be  invertible  for all integers $s\geq0$. Then the following recursion formula holds  true for  $\Gamma_{\mathcal{B}}^{H}$:
\begin{align}
&\Gamma_{\mathcal{B}}^{H}[(A; x), B, B'+sI; C,C', C''; z_1, z_2, z_3]\notag\\
&=\Gamma_{\mathcal{B}}^{H}[(A; x), B, B'; C,C', C''; z_1, z_2, z_3]\notag\\&\quad+ z_2{B}{C'}^{-1}\Big[\sum_{k=1}^{s}\Gamma_{\mathcal{B}}^{H}[(A; x), B+I, B'+kI; C, C'+I, C''; z_1, z_2, z_3]\Big]\notag\\
&\quad+ z_3{A}{C''}^{-1}\Big[\sum_{k=1}^{s}\Gamma_{\mathcal{B}}^{H}[(A+I; x), B, B'+kI; C, C', C''+I; z_1, z_2, z_3]\Big].\label{sr11}
\end{align}
Furthermore, if  $B'-kI$ are invertible for integers $k\leq s$, then 
\begin{align}&\Gamma_{\mathcal{B}}^{H}[(A; x), B, B'-sI; C,C', C''; z_1, z_2, z_3]\notag\\
&=\Gamma_{\mathcal{B}}^{H}[(A; x), B, B'; C,C', C''; z_1, z_2, z_3]\notag\\&\quad- z_2{B}{C'}^{-1}\Big[\sum_{k=0}^{s-1}\Gamma_{\mathcal{B}}^{H}[(A; x), B+I, B'-kI; C, C'+I, C''; z_1, z_2, z_3]\Big]\notag\\
&\quad- z_3{A}{C''}^{-1}\Big[\sum_{k=0}^{s-1}\Gamma_{\mathcal{B}}^{H}[(A+I; x), B, B'-kI; C, C', C''+I; z_1, z_2, z_3]\Big].\label{sr12}
\end{align}
where $A$, $B$, $B'$, $C$, $C'$ and $C''$ are  positive stable and commutative matrices in $\mathbb{C}^{r\times r}$ such that $C+kI$, $C'+kI$ and $C''+kI$ are invertible  matrices for all integers $k\geq 0$.
\end{theorem}
Other recursion formulas for the matrix functions $\Gamma_{\mathcal{B}}^{H}$ about the matrix parameter  $B'$ can be obtained as follows:
\begin{theorem}Let $B'+sI$  be  invertible  for all integers $s\geq0$. Then the following recursion formula holds  true for  $\Gamma_{\mathcal{B}}^{H}$:
\begin{align}
&\Gamma_{\mathcal{B}}^{H}[(A; x), B, B'+sI; C, C', C''; z_1, z_2, z_3]\notag\\
&=\sum_{k_1+k_2\leq s}^{}{s\choose k_1, k_2}z_{2}^{k_1} z_{3}^{k_2}{(A)_{k_2}}(B)_{k_1}(C')^{-1}_{k_1} (C'')^{-1}_{k_2} \,\notag\\
&\quad\times\Big[\Gamma_{\mathcal{B}}^{H}[(A+k_2 I; x), B+k_1I, B'+(k_1+k_2)I; C, C'+k_1I, C''+k_2 I; z_1, z_2, z_3]\Big].\label{s3eqh1}
\end{align}
Furthermore, if  $B'-kI$ are invertible for integers $k\leq s$, then 
\begin{align}&\Gamma_{\mathcal{B}}^{H}[(A; x), B, B'-sI; C, C', C''; z_1, z_2, z_3]\notag\\
&=\sum_{k_1+k_2\leq s}^{}{s\choose k_1, k_2}(-z_{2})^{k_1} (-z_{3})^{k_2}{(A)_{k_2}}(B)_{k_1}(C')^{-1}_{k_1} (C'')^{-1}_{k_2}\,\notag\\
&\quad\times\Big[\Gamma_{\mathcal{B}}^{H}[(A+k_2 I; x), B+k_1I, B'; C, C'+k_1I,  C''+k_2 I; z_1, z_2, z_3]\Big],\label{s3eqh2}
\end{align}
where $A$, $B$, $B'$, $C$, $C'$ and $C''$ are positive stable and commutative matrices in $\mathbb{C}^{r\times r}$ such that $C+kI$, $C'+kI$ and $C''+kI$  are  invertible  matrices for all integers $k\geq 0$.
\end{theorem}
\begin{theorem}
Let $C-sI$, $C'-sI$ and $C''-sI$ be invertible matrices for all $s\geq0$. Then the following recursion formulas hold  true for  $\Gamma_{\mathcal{B}}^{H}$:
\begin{align}
&\Gamma_{\mathcal{B}}^{H}[(A; x), B, B'; C-sI, C', C'';z_1, z_2, z_3]\notag\\
&=\Gamma_{\mathcal{B}}^{H}[(A; x), B, B'; C, C', C''; z_1, z_2, z_3]\notag\\&\quad+ z_1 AB\Big[\sum_{k=1}^{s} \Gamma_{\mathcal{B}}^{H}[(A+I; x), B+I, B'; C+(2-k)I, C', C''; z_1, z_2, z_3]\notag\\&\quad\times{(C-kI)^{-1}(C-(k-1)I)^{-1}}\Big];\label{szz1}
\end{align}
\begin{align}
&\Gamma_{\mathcal{B}}^{H}[(A; x), B, B'; C, C'-sI, C''; z_1, z_2, z_3]\notag\\
&=\Gamma_{\mathcal{B}}^{H}[(A; x), B, B'; C, C', C''; z_1, z_2, z_3]\notag\\&\quad+ z_2 B B'\Big[\sum_{k=1}^{s} \Gamma_{\mathcal{B}}^{H}[(A; x), B+I, B'+I;\,C, C'+(2-k)I, C''; z_1, z_2, z_3]\,\notag\\
&\quad\times{(C'-kI)^{-1}(C'-(k-1)I)^{-1}}\Big];\label{sph2}
\end{align}
\begin{align}
&\Gamma_{\mathcal{B}}^{H}[(A; x), B, B'; C, C', C''-sI; z_1, z_2, z_3]\notag\\
&=\Gamma_{\mathcal{B}}^{H}[(A; x), B, B'; C, C', C''; z_1, z_2, z_3]\notag\\&\quad+ z_3\, A B'\Big[\sum_{k=1}^{s} \Gamma_{\mathcal{B}}^{H}[(A+I; x), B, B'+I;\,C, C', C''+(2-k)I; z_1, z_2, z_3]\,\notag\\
&\quad\times{(C''-kI)^{-1}(C''-(k-1)I)^{-1}}\Big].\label{sph3}
\end{align}
where $A$, $B$, $B'$, $C$, $C'$ and $C''$ are positive stable and commutative matrices in $\mathbb{C}^{r\times r}$ such that $C+kI$, $C'+kI$ and $C''+kI$  are invertible  matrices for all integers $k\geq 0$.
\end{theorem}
\begin{theorem}Let $A$, $B$, $B'$, $C$, $C'$ and $C''$ be  positive stable  and commutative matrices  in  $\mathbb{C}^{r\times r}$ such that  $C+kI$, $C'+kI$ and  $C''+kI$  are invertible for all integers $k\geq0$. Then the following  derivative formulas  for $\Gamma_{\mathcal{B}}^{H}$ in (\ref{s2eq2}) hold true: 
\begin{align}
&\frac{\partial^{m}}{\partial z_{1}^{m}}\Gamma_{\mathcal{B}}^{H}[(A; x), B, B'; C,C', C''; z_1, z_2, z_3]\notag\\&= (A)_{m}(B)_{m} (C)^{-1}_{m}\Big[\Gamma_{\mathcal{B}}^{H}[(A+mI; x), B+mI, B'; C+mI,C', C''; z_1, z_2, z_3]\Big];\label{sd1}\\
&\frac{\partial^{m+n}}{\partial z_{1}^{m}\partial z_{2}^{n}}\Gamma_{\mathcal{B}}^{H}[(A; x), B, B'; C,C', C''; z_1, z_2, z_3]\notag\\&= (A)_{m}(B)_{m+n}(B')_{n}(C)^{-1}_{m} (C')^{-1}_{n}\notag\\& \Big[\Gamma_{\mathcal{B}}^{H}[(A+mI; x), B+(m+n)I, B'+nI; C+mI,C'+nI, C''; z_1, z_2, z_3]\Big];\label{sd2}\\
&\frac{\partial^{m+n+p}}{\partial z_{1}^{m}\partial z_{2}^{n}\partial z_{3}^{p}}\Gamma_{\mathcal{B}}^{H}[(A; x), B, B'; C,C', C''; z_1, z_2, z_3]\notag\\&= (A)_{m+p}(B)_{m+n}(B')_{n+p}(C)^{-1}_{m} (C')^{-1}_{n} (C'')^{-1}_{p}\notag\\& \Big[\Gamma_{\mathcal{B}}^{H}[(A+(m+p)I; x), B+(m+n)I, B'+(n+p)I; C+mI,C'+nI, C''+pI; z_1, z_2, z_3]\Big].\label{sd3}
\end{align}
\end{theorem}
\begin{theorem} Let  $A$, $B$, $B'$, $C$, $C'$ and $C''$ be  positive stable  and commutative matrices in  $\mathbb{C}^{r\times r}$ such that  $C+kI$, $C'+kI$ and  $C''+kI$ are invertible for all integer $k\geq0$. Then the following  recurrence relation  for $\Gamma_{\mathcal{B}}^{H}$ in (\ref{s2eq2}) holds true: 
\begin{align}&\Gamma_{\mathcal{B}}^{H}[(A; x), B, B'; C,C', C''; z_1, z_2, z_3]=\Gamma_{\mathcal{B}}^{H}[(A; x), B, B'; C-I,C', C''; z_1, z_2, z_3]\notag\\
&-ABC^{-1}(C-I)^{-1} \Gamma_{\mathcal{B}}^{H}[(A+I; x), B+I, B'; C+I,C', C''; z_1, z_2, z_3].
\end{align}
\end{theorem}
\section{The  incomplete Srivastava's  triple  hypergeometric matrix  functions $\gamma_{\mathcal{C}}^{H}$ and $\Gamma_{\mathcal{C}}^{H}$}
In this section, we introduce the incomplete Srivastava's  triple  hypergeometric matrix  functions $\gamma_{\mathcal{C}}^{H}$ and $\Gamma_{\mathcal{C}}^{H}$ as follows: 
\begin{align}
&\gamma_{\mathcal{C}}^{H}[(A; x), B, B'; C; z_1, z_2, z_3]\notag\\&=\sum_{m, n, p\geq 0}(A;x)_{m+p} (B)_{m+n} (B')_{n+p} (C)^{-1}_{m+n+p}\frac{z_{1}^{m}z_2^{n}z_3^{p}}{m! n! p!},\label{ps2eq1}\\
&\Gamma_{\mathcal{C}}^{H}[(A; x), B, B'; C; z_1, z_2, z_3]\notag\\&=\sum_{m, n, p\geq 0}[A;x]_{m+p} (B)_{m+n} (B')_{n+p} (C)^{-1}_{m+n+p}\frac{z_{1}^{m}z_2^{n}z_3^{p}}{m! n! p!}.\label{ps2eq2}
\end{align}
where $A$, $B$, $B'$  and $C$ are  positive stable and commutative matrices in  $\mathbb{C}^{r\times r}$ such that $C+kI$  is invertible for all integers $k\geq0$.\\
From (\ref{1eq9}), we have the following decomposition formula
\begin{align}
&\gamma_{\mathcal{C}}^{H}[(A; x), B, B'; C ; z_1, z_2, z_3]+ \Gamma_{\mathcal{C}}^{H}[(A; x), B, B'; C ; z_1, z_2, z_3]\notag\\&= H_{\mathcal{C}}[A, B, B'; C ; z_1, z_2, z_3].\label{ps2eq3}
\end{align}
where $H_{\mathcal{C}}[A, B, B'; C ; z_1, z_2, z_3]$ is the Srivastava's  triple  hypergeometric matrix  functions \cite{RD1}.

\begin{remark}
It is quite evident to note that the special cases of (\ref{ps2eq1}) and (\ref{ps2eq2}) when $z_2 = 0$ reduce to the
known incomplete families of the first  Appell hypergeometric  matrix functions \cite{A1}. Also, the special
cases of (\ref{ps2eq1}) and (\ref{ps2eq2}) when $z_2 = 0$ and  $z_3 = 0$ or $z_1 = 0$ are seen to yield the known incomplete families of Gauss hypergeometric  matrix functions \cite{Ab}.\end{remark}

To discuss the properties and characteristics of $\gamma_{\mathcal{C}}^{H}[(A; x), B, B'; C ; z_1, z_2, z_3]$, it is sufficient to determine the properties of $\Gamma_{\mathcal{C}}^{H}[(A; x), B, B'; C ; z_1, z_2, z_3]$ according to the decomposition formula (\ref{s2eq3}).We omit the proofs of the given below theorems.
\begin{theorem}Let $A$, $B$, $B'$   and $C$ be  positive stable and commutative matrices in  $\mathbb{C}^{r\times r}$ such that $C+kI$  is invertible for all integers $k\geq0.$ Then the function defined by 
${\mathcal{T}_3}={\mathcal{T}_3}(z_1, z_2, z_3) =\gamma_{\mathcal{C}}^{H}[(A; x), B, B'; C ; z_1, z_2, z_3]+\Gamma_{\mathcal{C}}^{H}[(A; x), B, B'; C ; z_1, z_2, z_3]$   satisfies the following system of partial differential equations:
\begin{align}
&\Big[z_1\frac{\partial}{\partial z_1}{(z_1\frac{\partial}{\partial z_1}+z_2\frac{\partial}{\partial z_2}+z_3\frac{\partial}{\partial z_3}+C-I)- z_1(z_1\frac{\partial}{\partial z_1}+z_3\frac{\partial}{\partial z_3}+A)(z_1\frac{\partial}{\partial z_1}+z_2\frac{\partial}{\partial z_2}+B)}\Big] \mathcal{T}_3=O,\label{p2eq4}\\
&\Big[z_2\frac{\partial}{\partial z_2}{(z_1\frac{\partial}{\partial z_1}+z_2\frac{\partial}{\partial z_2}+z_3\frac{\partial}{\partial z_3}+C-I)- z_2(z_2\frac{\partial}{\partial z_2}+z_1\frac{\partial}{\partial z_1}+B)(z_2\frac{\partial}{\partial z_2}+z_3\frac{\partial}{\partial z_3}+B')}\Big] \mathcal{T}_3=O,\label{pm1}\\
&\Big[z_3\frac{\partial}{\partial z_3}{(z_1\frac{\partial}{\partial z_1}+z_2\frac{\partial}{\partial z_2}+z_3\frac{\partial}{\partial z_3}+C-I)- z_3(z_1\frac{\partial}{\partial z_1}+z_3\frac{\partial}{\partial z_3}+A)(z_2\frac{\partial}{\partial z_2}+z_3\frac{\partial}{\partial z_3}+B')}\Big] \mathcal{T}_3=O.\label{ps2eq5}
\end{align}
\end{theorem}
\begin{theorem} Let $A$, $B$, $B'$ and $C$ be  positive stable and commutative  matrices in  $\mathbb{C}^{r\times r}$ such that $A+kI$, $B+kI$ and  $C+kI$  are  invertible for all integers $k\geq0$. Then the following integral representation for $\Gamma_{\mathcal{C}}^{H}$ in (\ref{ps2eq2}) holds true:
\begin{align}
&\Gamma_{\mathcal{C}}^{H}[(A; x), B, B'; C ; z_1, z_2, z_3]= \Gamma^{-1}(A) \Gamma^{-1}(B)\notag\\&\times \Big[\int_{x}^{\infty} \int_{0}^{\infty}e^{-t-s}\, t^{A-I} s^{B-I}  \Phi_{3}(B'; C; z_2 s+ z_3 t, z_1 st) dt ds\Big],\label{ps2eq6}
\end{align}where $\Phi_{3}$ is the Humbert's matrix function  defined by \cite{MA, SZ}
\begin{align}\Phi_{3}(B'; C ; z_1, z_2)=\sum_{m, n \geq 0}^{}(B')_{m} (C)^{-1}_{m+n} \frac{z_{1}^{m} z_{2}^{n}}{m! n!}.
\end{align} 
\end{theorem}
\begin{theorem}Let $A$, $B$, $B'$ and $C$ be  positive stable and commutative matrices in  $\mathbb{C}^{r\times r}$ such that $A+kI$, $B+kI$, $B'+kI$ and  $C+kI$ are  invertible for all integers $k\geq0$. Then the following triple  integral representation for $\Gamma_{\mathcal{C}}^{H}$ in (\ref{ps2eq2}) holds true:
\begin{align}
&\Gamma_{\mathcal{C}}^{H}[(A; x), B, B'; C ; z_1, z_2, z_3]\notag\\
&= \Gamma^{-1}(A) \Gamma^{-1}(B) \Gamma^{-1}(B')\Big[\int_{x}^{\infty}\int_{0}^{\infty} \int_{0}^{\infty}e^{-t-s-u}\, t^{A-I} s^{B-I}  u^{B'-I} \notag\\& \times {_{0}F_{1}}(-; C ; z_1 st+z_2 us+z_3ut) dt ds du\Big].\label{psx1}
\end{align}
\end{theorem}
\begin{theorem}\label{h1}Let $B'+sI$  be  invertible  for all integers $s\geq0$. Then the following recursion formula holds  true for  $\Gamma_{\mathcal{C}}^{H}$:
\begin{align}
&\Gamma_{\mathcal{C}}^{H}[(A; x), B, B'+sI; C ; z_1, z_2, z_3]\notag\\
&=\Gamma_{\mathcal{C}}^{H}[(A; x), B, B'; C ; z_1, z_2, z_3]\notag\\&\quad+ z_2{B}{C}^{-1}\Big[\sum_{k=1}^{s}\Gamma_{\mathcal{C}}^{H}[(A; x), B+I, B'+kI;  C+I ; z_1, z_2, z_3]\Big]\notag\\
&\quad+ z_3{A}{C}^{-1}\Big[\sum_{k=1}^{s}\Gamma_{\mathcal{C}}^{H}[(A+I; x), B, B'+kI;  C+I; z_1, z_2, z_3]\Big].\label{psr11}
\end{align}
Furthermore, if  $B'-kI$ are invertible for integers $k\leq s$, then 
\begin{align}&\Gamma_{\mathcal{C}}^{H}[(A; x), B, B'-sI; C ; z_1, z_2, z_3]\notag\\
&=\Gamma_{\mathcal{C}}^{H}[(A; x), B, B'; C; z_1, z_2, z_3]\notag\\&\quad- z_2{B}{C}^{-1}\Big[\sum_{k=0}^{s-1}\Gamma_{\mathcal{C}}^{H}[(A; x), B+I, B'-kI; C+I ; z_1, z_2, z_3]\Big]\notag\\
&\quad- z_3{A}{C}^{-1}\Big[\sum_{k=0}^{s-1}\Gamma_{\mathcal{C}}^{H}[(A+I; x), B, B'-kI;  C+I; z_1, z_2, z_3]\Big].\label{psr12}
\end{align}
where $A$, $B$, $B'$ and $C$ are  positive stable and commutative matrices in $\mathbb{C}^{r\times r}$ such that $C+kI$ is invertible  matrix for all integers $k\geq 0$.
\end{theorem}
Other recursion formulas for the matrix functions $\Gamma_{\mathcal{C}}^{H}$ about the matrix parameter  $B'$ can be obtained as follows:
\begin{theorem}Let $B'+sI$  be  invertible  for all integers $s\geq0$. Then the following recursion formula holds  true for  $\Gamma_{\mathcal{C}}^{H}$:
\begin{align}
&\Gamma_{\mathcal{C}}^{H}[(A; x), B, B'+sI; C ; z_1, z_2, z_3]\notag\\
&=\sum_{k_1+k_2\leq s}^{}{s\choose k_1, k_2}z_{2}^{k_1} z_{3}^{k_2}{(A)_{k_2}}(B)_{k_1}(C)^{-1}_{k_1+k_2}  \,\notag\\
&\quad\times\Big[\Gamma_{\mathcal{C}}^{H}[(A+k_2 I; x), B+k_1I, B'+(k_1+k_2)I;  C+(k_1+k_2)I; z_1, z_2, z_3]\Big].\label{ps3eqh1}
\end{align}
Furthermore, if  $B'-kI$ are invertible for   integers $k\leq s$, then 
\begin{align}&\Gamma_{\mathcal{C}}^{H}[(A; x), B, B'-sI; C ; z_1, z_2, z_3]\notag\\
&=\sum_{k_1+k_2\leq s}^{}{s\choose k_1, k_2}(-z_{2})^{k_1} (-z_{3})^{k_2}{(A)_{k_2}}(B)_{k_1}(C)^{-1}_{k_1+k_2}\,\notag\\
&\quad\times\Big[\Gamma_{\mathcal{C}}^{H}[(A+k_2 I; x), B+k_1I, B';  C+(k_1+k_2)I ; z_1, z_2, z_3]\Big],\label{s3eqh2}
\end{align}
where $A$, $B$, $B'$ and $C$ are positive stable and commutative matrices in $\mathbb{C}^{r\times r}$ such that $C+kI$ is  invertible  matrices for all integers $k\geq 0$.
\end{theorem}
\begin{theorem}
Let $C-sI$  be invertible for all integers $s\geq0$. Then the following recursion formulas hold  true for  $\Gamma_{\mathcal{C}}^{H}$:
\begin{align}
&\Gamma_{\mathcal{C}}^{H}[(A; x), B, B'; C-sI; z_1, z_2, z_3]\notag\\
&=\Gamma_{\mathcal{C}}^{H}[(A; x), B, B'; C; z_1, z_2, z_3]\notag\\&\quad+ z_1 AB\Big[\sum_{k=1}^{s} \Gamma_{\mathcal{C}}^{H}[(A+I; x), B+I, B'; C+(2-k)I; z_1, z_2, z_3]\notag\\&\quad\times{(C-kI)^{-1}(C-(k-1)I)^{-1}}\Big]\notag\\&\quad+ z_2 B B'\Big[\sum_{k=1}^{s} \Gamma_{\mathcal{C}}^{H}[(A; x), B+I, B'+I;  C+(2-k)I ; z_1, z_2, z_3]\,\notag\\
&\quad\times{(C-kI)^{-1}(C-(k-1)I)^{-1}}\Big]\notag\\
&+ z_3 A B'\Big[\sum_{k=1}^{s} \Gamma_{\mathcal{C}}^{H}[(A+I; x), B, B'+I;  C+(2-k)I ; z_1, z_2, z_3]\notag\\
&\quad\times{(C-kI)^{-1}(C-(k-1)I)^{-1}}\Big],\label{pph2}
\end{align}
where $A$, $B$, $B'$  and $C$ are positive stable and commutative matrices in $\mathbb{C}^{r\times r}$ such that $C+kI$ is invertible  matrix for all integers $k\geq 0$.
\end{theorem}
\begin{theorem}Let $A$, $B$, $B'$ and $C$ be  positive stable  and commutative matrices  in  $\mathbb{C}^{r\times r}$ such that   $C+kI$  is invertible for all integers $k\geq0$. Then the following  derivative formulas  for $\Gamma_{\mathcal{C}}^{H}$ in (\ref{ps2eq2}) hold true: 
\begin{align}
&\frac{\partial^{m}}{\partial z_{1}^{m}}\Gamma_{\mathcal{C}}^{H}[(A; x), B, B'; C; z_1, z_2, z_3]\notag\\&= (A)_{m}(B)_{m} (C)^{-1}_{m}\Big[\Gamma_{\mathcal{C}}^{H}[(A+mI; x), B+mI, B'; C+mI; z_1, z_2, z_3]\Big];\label{sd1}\\
&\frac{\partial^{m+n}}{\partial z_{1}^{m}\partial z_{2}^{n}}\Gamma_{\mathcal{C}}^{H}[(A; x), B, B'; C ; z_1, z_2, z_3]\notag\\&= (A)_{m}(B)_{m+n}(B')_{n}(C)^{-1}_{m+n}\notag\\& \Big[\Gamma_{\mathcal{C}}^{H}[(A+mI; x), B+(m+n)I, B'+nI; C+(m+n)I; z_1, z_2, z_3]\Big];\label{sd2}\\
&\frac{\partial^{m+n+p}}{\partial z_{1}^{m}\partial z_{2}^{n}\partial z_{3}^{p}}\Gamma_{\mathcal{C}}^{H}[(A; x), B, B'; C ; z_1, z_2, z_3]\notag\\&= (A)_{m+p}(B)_{m+n}(B')_{n+p}(C)^{-1}_{m+n+p}\notag\\& \Big[\Gamma_{\mathcal{C}}^{H}[(A+(m+p)I; x), B+(m+n)I, B'+(n+p)I; C+(m+n+p)I; z_1, z_2, z_3]\Big].\label{sd3}
\end{align}
\end{theorem}
\begin{remark}The results in this article generalize the corresponding results of
Choi \textit{et. al.}\cite{jc, jc1} to the matrix case. For $x=0$, the incomplete Srivastava's  triple  hypergeometric matrix  functions reduces to the Srivastava's  triple  hypergeometric matrix  function. Therefore,  taking $x=0$, the obtained results for the incomplete Srivastava's  triple  hypergeometric matrix  functions  reduce to the  results for the Srivastava's  triple  hypergeometric matrix  functions.
\end{remark}
{\bf Acknowledgements} The author would like to thank the referees for their valuable comments and suggestions
which have led to a better presentation of the paper. This work is inspired by and is dedicated to Professor Vivek Sahai,  Department of Mathematics and Astronomy, University of
Lucknow.

\end{document}